\documentclass{amsart}

\usepackage{amsfonts, amsmath, amsthm, graphicx}
\usepackage{epsfig, psfrag}
\usepackage{xfrac}

\theoremstyle{definition}

\theoremstyle{remark}

\numberwithin{equation}{section}

\swapnumbers \theoremstyle{plain}

\newtheorem{thm}{Theorem}[section]

\newtheorem{lem}[thm]{Lemma}

\newtheorem{cor}[thm]{Corollary}

\newtheorem{prop}[thm]{Proposition}

\theoremstyle{remark}

\theoremstyle{definition}
\newtheorem{defn}[thm]{Definition}

\newcommand{\T}{\mathcal T}
\newcommand{\C}{\mathcal C}

\renewcommand{\S}{{\mathcal S}}

\newcommand{\bdy}{\partial}
\newcommand{\bbb}{\mathbb}
\newcommand{\rppp}{\mathbb{R}P^3}

\newcommand{\td}{\tilde}
\newcommand{\open}[1]{\stackrel{\circ}{#1}}
\newcommand{\ve}{\varepsilon}

\newcommand{\abs}[1]{\lvert#1\rvert}

\begin{document}

\title{Efficient triangulations and boundary slopes}
\author{Birch Bryant}
\address{Department of Mathematics, Oklahoma State University,
Stillwater, OK 74078}
\email{jbirch.bryant@okstate.edu}
\thanks{The first author was partially supported by
The Grayce B. Kerr Foundation}
\author{William Jaco}
\address{Department of Mathematics, Oklahoma State University,
Stillwater, OK 74078}

\email{william.jaco@okstate.edu}
\thanks{The second author was partially supported by NSF/DMS Grants 9971719 and 0204707,
The Grayce B. Kerr Foundation, The American Institute of Mathematics
(AIM), and and The Visiting Research Scholar Program at University
of Melbourne (Australia)}

\author{J.~Hyam Rubinstein}
\address{Department of Mathematics and Statistics,
University of Melbourne, Parkville, 
VIC 3010, Australia}
\email{rubin@ms.unimelb.edu.au}
\thanks{The third author was partially supported by The Australian Research
Council and The Grayce B. Kerr Foundation.}

 \date{\today}

\maketitle

\section{Introduction}  In \cite{jac-rub-0eff} it is shown that for $M \ne \mathbb{B}^3$, a compact, irreducible, $\bdy$--irreducible, and an-annular $3$-manifold with non-empty boundary,  any triangulation $\T$ of $M$ can be modified by crushing along a normal surface in $\T$ to an ideal triangulation $\T^*$ of the interior of $M$, $\open{M}$. The ideal triangulation $\T^*$ is constructed from a proper subset of the tetrahedra and face identifications of $\T$, resulting in $|\T^*|<|\T|$. Conversely, in \cite{jac-rub-inflate} it is shown that for any ideal triangulation $\T^*$ of $\open{M}$ there is an inflation modifying $\T^*$ to a minimal-vertex, normal-boundary triangulation $\T_{\Lambda}$ of $M$. Such an inflation is constructed from the tetrahedra and face identifications of $\T^*$ guided by a frame, $\Lambda$, in the vertex-linking surface(s) of $\T^*$; furthermore, the triangulation $\T_\Lambda$ can be crushed along its boundary-linking normal surface(s) returning to the triangulation $\T^*$ of $\open{M}$. Inflation triangulations of $\T^*$ may vary according to the chosen frame; however, each has a boundary-linking normal surface and admits a unique combinatorial crushing along the boundry-linking normal surface  back to the same ideal triangulation $\T^*$ of $\open{M}$.  

In Section 3 we review crushing triangulations along normal surfaces and inflating ideal triangulations.   Since an inflation is defined in terms of a combinatorial crushing, this enables us to establish very strong relationships between ideal triangulations and their inflations.  In particular, in Theorem \ref{bijection-ideal-inflate}  we establish a one-one correspondence between the closed normal surfaces  in an ideal triangulation  and the closed normal surface in any of its inflations; furthermore, surfaces corresponding under this bijection  are homeomorphic. In particular, starting with an ideal triangulation, $\T^*$, of the interior of a compact 3-manifold, $M$, with boundary, if $\T^*$ is $0$-efficient, $1$-efficient, or end-efficient, then any inflation of $\T^*$ is $0$-efficient, $1$-efficient, or $\bdy$-efficient, respectively. In the last case  the inflation triangulation is also annular-efficient if the manifold $M$ is an-annular. 

In Section 4, and similar to results about $0$-efficient triangulations \cite{jac-rub-0eff}, we define and study annular-efficient triangulations, normal-boundary triangulations, and boundary-linking normal surfaces and introduce boundary-efficient triangulations and end-efficient ideal triangulations. We characterize topological conditions necessary and sufficient for each of these triangulations and provide an algorithm showing for a manifold $M$, satisfying these conditions and distinct from $\mathbb{B}^3$, then any triangulation of $M$ or ideal triangulation of $\open{M}$ may be modified to a boundary-efficient triangulation of $M$ or an end-efficient ideal triangulation of $\open{M}$, respectively.  

Using these methods for modifying any given triangulation, we are able to add to the study and understanding of  boundary slopes for interesting surfaces in compact 3-manifolds with boundary.  There are a number of results in the literature on boundary slopes for incompressible and $\bdy$-incompressible surfaces in knot manifolds. The major question is to whether the isotopy classes (slopes) for the boundaries of such surfaces in the exterior of a knot in $S^3$ is a finite set. This had been shown to be true for certain knots (torus knots, the figure-eight knot, 2-bridge knots, and alternating knots) when A. Hatcher \cite{hat-82} proved that for a compact 3-manifold $M$ with a single incompressible torus boundary, there are just a finite number of slopes for boundary curves of incompressible and $\bdy$-incompressible surfaces in $M$.  For manifolds with a single torus  boundary, this was generalized to normal surfaces in any minimal-vertex triangulation of $M$ \cite{jac-sed-dehn} and, hence, true for incompressible and $\bdy$-incompressible surfaces in $M$. Also, in \cite{jac-rub-sed}  there is a special case to one of our results in this paper for a compact, irreducible, $\bdy$-irreducible and an-annular 3-manifold having all components of its boundary tori; namely,  there are only finitely many boundary slopes for normal surfaces (and, hence, for  incompressible and $\bdy$-incompressible surfaces) of a bounded Euler characteristics.  In Section 5, we extend this latter result to manifolds with arbitrary genus boundary components that admit an annular-efficient triangulation to prove there are only a finite number of boundary slopes for normal surfaces of a bounded Euler characteristic; hence,  for a compact, irreducible, $\bdy$--irreducible, and an-annular 3-manifolds there are only a finite number of boundary slopes possible for incompressible and $\bdy$--incompressible surfaces having a bounded Euler characteristic.  

We wish to thank David Bachman and Saul Schleimer for many conversations and insights; they studied similar problems and use the terminology $\sfrac{1}{2}$-efficient triangulation for what we call annular-efficient.  In \cite{lack-taut} and \cite{lack-HS-genus} M. Lackenby studies angle structures on ideal triangulations of the interior of compact, orientable, simple (irreducible, $\bdy$-irreducible, an-annular and atoridal)  3-manifolds.   The latter is most relevant to this paper.   We discuss the relationship of these results and ours  in a brief Summary at the end of this paper.

\section{Triangulations and Normal Surfaces} We continue with the
use of (pseudo) triangulations and ideal triangulations as in  \cite{jac-rub-0eff}.

 If
$\mathbf{\td{\Delta}}$ is a pairwise disjoint collection of oriented
tetrahedra and $\mathbf{\Phi}$ is a family of orientation-reversing, 
affine face identifications of the tetrahedra in $\mathbf{\td{\Delta}}$,
then the identification space $X
=\mathbf{\td{\Delta}}/\mathbf{\Phi}$ is a 3-complex and is a
3--manifold at each point except possibly at the vertices. If $X$ is
a manifold, we denote the collection of tetrahedra and the face
identifications by a single symbol $\T$ and say $\T$ is a
\emph{triangulation} of the manifold $X$. If $X$ is not a manifold,
then $X\setminus\{vertices\}$ is the interior of a compact
$3$--manifold, $M$,  with boundary and we say $\T$ is an \emph{ideal
triangulation} of $\open{M}$, the interior of $M$; in this case we also say that $X$ is a \emph{pseudo-manifold} and $\T$ is an ideal triangulation of $X$. At each vertex of a tetrahedron there is a properly embedded triangle separating that vertex from the face of the tetrahederon opposite the vertex; the collection of all such triangles, one for each vertex of each tetrahedron, form a properly embedded surface. A component of this surface is call a {\it vertex-linking surface} for the associated vertex of the triangulation.   For an ideal triangulation, the
image of a vertex of a tetrahedron in $\mathbf{\td{\Delta}}$ is
called an {\it ideal vertex} and its {\it index} is the genus of its
vertex-linking surface. For an ideal triangulation we will always
assume the index of each vertex is $\ge 1$.

For our triangulations, the simplexs of $\mathbf{\td{\Delta}}$ are
not necessarily embedded in $X$; however, the interior of each
simplex is embedded. We call the image in $X$ of a tetrahedron,
face, or edge in $\mathbf{\td{\Delta}}$, a tetrahedron, face, or
edge, respectively. For a tetrahedron $\Delta$ in $X$, there is precisely one
tetrahedron $\td{\Delta}$ in $\mathbf{\td{\Delta}}$ that projects to
$\Delta$, called the {\it lift of $\Delta$}. For a face $\sigma$ in
$X$, there are either one or two faces in $\mathbf{\td{\Delta}}$
that project to $\sigma$; if only one face projects to $\sigma$,
then $\sigma$ is in the boundary of $M$. If $e$ is an edge in $X$,
the number of edges in $\mathbf{\td{\Delta}}$ that project to $e$ is
the {\it index of $e$}.

See \cite{jac-rub-0eff} for more details regarding triangulations from
our point of view.

\subsection{Normal surfaces} If $M$ is a $3$--manifold and $\T$ is a
triangulation of $M$, we say the properly embedded surface $S$ in
$M$ is {\it normal} (with respect to $\T$) if for every tetrahedron
$\Delta$ in $\mathbf{\td{\Delta}}/\mathbf{\Phi}$,  the intersection
of $S$ with $\Delta$ lifts to a collection of normal triangles and
normal quadrilaterals in $\td{\Delta}$, the lift of $\Delta$.
Note
that since our tetrahedra have possible face identifications, the
intersection of a normal surface 
with a tetrahedron need not be a
normal triangle or a normal quadrilateral 
but might be one of these
with edge identifications.

We shall assume the reader is familiar with classical normal 
surface
theory, which carries over in all of our situations. In particular, if $\T$ is a triangulation or ideal triangulation of the manifold $M$, an isotopy of $M$ is called a {\it normal isotopy} (with respect to $\T$) if it is invariant on each simplex of $\T$. 

A triangulation of a compact $3$--manifold
with boundary is said to be a {\it normal boundary triangulation} or
to have a {\it normal boundary} if the frontier of a small regular
neighborhood of the boundary is normally isotopic to a normal
surface. In this case, we call the normal surface consisting of the
frontier of a small regular neighborhood of the boundary a {\it
boundary-linking surface}. Not all triangulations have a normal boundary; for
example, layered triangulations of handlebodies \cite{jac-rub-layered} contain no closed normal surfaces
and, hence, can not have a normal boundary and a 0-efficient triangulation of the $3$--cell (there are infinitely many) contains no normal 2-spheres and thus has no normal boundary..  

A properly embedded surface in a compact 3--manifold with boundary is said to be \emph{isotopic into $\bdy M$} if there is an isotopy of the surface through $M$ into $\bdy M$ keeping the boundary of the surface in $\bdy M$.  If the manifold is triangulated and a surface is  closed and normal, it is said to be \emph{normally isotopic into $\bdy M$}, if  the triangulation has normal boundary and the surface is normally isotopic to the boundary-linking surface. We are interested in triangulations in which the only closed, normal surface isotopic into the boundary is boundary-linking. 

A properly embedded annulus in a $3$--manifold is {\it essential} if
it is incompressible and $\bdy$-incompressible.  For a 3--manifold that is irreducible and $\bdy$-irreducible,  a properly embedded annulus is essential iff it is irreducible and not isotopic into $\bdy M$.  A compact
$3$--manifold is said to be {\it an-annular} if it has no properly
embedded, essential annuli.

\section{Basics of crushing and inflating triangulations}\label{crushing-inflations}
\subsection{Crushing triangulations along normal surfaces}  In \cite{jac-rub-0eff} we introduced the procedure of ``crushing a
triangulation along a normal surface."  Details may be reviewed there, as well as in  \cite{jac-rub-inflate}, where the details apply more directly to our situation in this work.

Suppose
$\T$ is a triangulation of the compact $3$--manifold $M$ or an ideal
triangulation of the interior of $M$. Suppose $S$ is a closed normal
surface in $M$, $X$ is the closure of a component of the complement
of $S$, and $X$ does not contain any of the vertices of $\T$. Since
$X$ does not contain any of the vertices of $\T$, the triangulation
$\T$ induces a particularly nice cell-decomposition on $X$, say $\mathcal{C}_X$, 
consisting of {\it truncated-tetrahedra, truncated-prisms, triangular
product blocks}, and {\it quadrilateral product blocks}.  See Figure
\ref{f-cell-decomp}. 

\begin{figure}[htbp]
            \psfrag{X}{$X$}
            \psfrag{s}{\small tetrahedron}
            \psfrag{f}{\small face}

             \psfrag{c}{\small crush}
            \psfrag{e}{\small edge}
            \psfrag{t}{\begin{tabular}{c}
          {\small truncated-tetrahedron}\\
            \end{tabular}}
            \psfrag{p}{\begin{tabular}{c}
            {\small truncated-prism}\\
            \end{tabular}}
            \psfrag{q}{\begin{tabular}{c}
          {\small triangular}\\
          {\small product block}\\
            \end{tabular}}
            \psfrag{r}{\begin{tabular}{c}
          {\small quadrilateral}\\
          {\small product block}\\
            \end{tabular}}

        \vspace{0 in}
        \begin{center}
            \includegraphics[width=3.5 in]{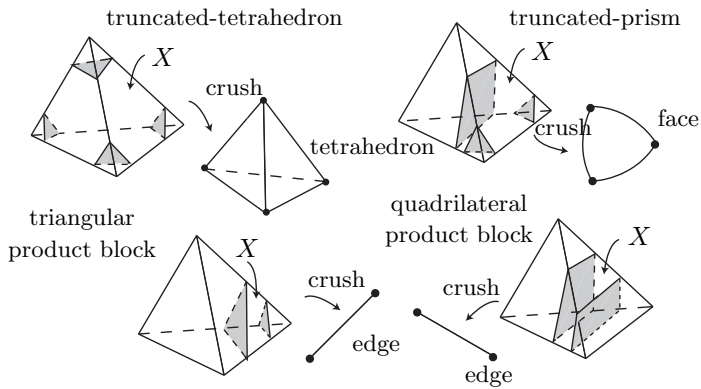}
        \caption{Cells in induced cell-decomposition $\mathcal{C}_X$ of $X$ and their crushing
         to tetrahedra, faces, and edges in an ideal triangulation of $\open{X}$.}
        \label{f-cell-decomp}
        \end{center}

\end{figure}]

The boundary of each $3$--cell in $\C_X$ has an induced cell
decomposition in which some of the cells are in $S$ and some are
not. The edges and faces in the decomposition $\C_X$ are called {\it
horizontal} if their interiors are in $S$ and {\it vertical} if
their interiors are not in $S$. The quadrilateral vertical
$2$--cells are called {\it trapezoids}; there are two in a
truncated-prism, three in a triangular block, and four in a
quadrilateral block. The non-trapezoidal vertical $2$--cells are in
truncated-prisms and truncated-tetrahedra and are hexagons.

We define $\mathbb{P}(\C_X)$ as the union, $\mathbb{P}(\C_X)=$ $ \{$vertical edges of $\C_X  \}$ $ \cup \{$trapezoids$\}\cup\{$triangular
blocks$\}\cup\{$quadrilateral blocks$\}$.  Each component of $\mathbb{P}(\C_X)$ is an $I$--bundle.  Suppose each component of $\mathbb{P}(\C_X)$ is a product $I$ bundle. Then a component of $\mathbb{P}(\C_X)$ is a product $\mathbb{P}_i = K_i\times I$,  where $K_i^\ve =K_i\times \ve, \ve =0, 1$, and $K_i\times 0$ and $K_i\times 1$ are isomorphic  subcomplexes in the induced normal cell decomposition on $S$, $i=1,2,\ldots,k$, $k$ being the number of components of $\mathbb{P}(\C_X)$.  In this situation, we call $\mathbb{P}(\C_X)$ the {\it combinatorial product for $\C_X$}.  If $\mathbb{P}(\C_X) \not= X$ and each $K_i$ is a simply
connected planar complex (hence, it is cell-like), we say $\bbb{P}(\C_X)$ is a {\it trivial combinatorial product}.  In applications, we do not always have things so nice and we need to modify $\mathbb{P}(\C_X)$ to an {\it induced product region for $X$}, denoted $\mathbb{P}(X)$.

Now, consider the truncated-prisms in $\C_X$. Each truncated-prism has
two hexagonal faces. In $\C_X$, these hexagonal faces are identified
via the face identifications of the given triangulation $\T$ to a
hexagonal face of a truncated-tetrahedron or to a hexagonal face of
truncated-prism. If we follow a sequence of such identifications
through hexagonal faces of truncated-prisms, we trace out a
well-defined arc that terminates at an identification with a
hexagonal face of a truncated-tetrahedron or possibly does not
terminate but forms a complete cycle through hexagonal faces of
truncated-prisms. We call a collection of
truncated-prisms identified in this way a {\it chain}. If a chain
ends in a truncated-tetrahedra, we say the chain {\it terminates};
otherwise, we call the chain a {\it cycle of truncated-prisms}.

Just as in \cite{jac-rub-inflate}, under appropriate conditions, we can construct an ideal triangulation of $\open{X}$ using a controlled crushing of the cells of $\C_X$.   In particular, to obtain the desired  ideal triangulation of $\open{X}$ it is sufficient that  
$X\ne \bbb{P}(\C_X)$ or in the more general case $X\ne \bbb{P}(X)$ (there are not too many product blocks) and there are no
cycles of truncated-prisms (there are not too many truncated-prisms).  As a result of the crushing,  each component of $S$ is
crushed to a point (distinct points for distinct components), all designated 
products are crushed to arcs and, in particular, the products $K_i\times I$ are crushed to arcs (edges) so that if
$K_i\times I$ is crushed to the edge $e_i$, then the crushing
projection coincides with the projection of $K_i\times I$ onto the
$I$ factor.  Vertical edges, trapezoids, and product
blocks in $\C_X$ are identified to edges in the ideal triangulation. Truncated-prisms becomes faces
and truncated-tetrahedra become  tetrahedra.  Consult \cite{jac-rub-0eff} and see Figure \ref{f-cell-decomp}.

The crushing is particularly nice in the case that $\bbb{P}(\C_X)$ is a trivial combinatorial product, $X\ne \bbb{P}(\C_X)$, and there are no cycles of truncated prisms.  In this case, suppose $\{\overline{\Delta}_1,\ldots,\overline{\Delta}_n\}$ denotes the
collection of truncated-tetrahedra in $\C_X$.  Each
truncated-tetrahedron in $\C_X$ has its triangular faces in $S$. If we
crush each such triangular face of a truncated-tetrahedron to a
point (for the moment, distinct points for each triangular face), we
get a tetrahedron. We use the notation $\td{\Delta}_i^*$ for the
tetrahedron coming from the truncated-tetrahedron
$\overline{\Delta}_i$ after identifying  the triangular faces of
$\overline{\Delta}_i$ to points.  Also as a consequence of this crushing of $S$, if $\overline{\sigma}_i$ is a
hexagonal face in $\overline{\Delta}_i$, then $\overline{\sigma}_i$
is identified to a triangular face, say $\td{\sigma}_i^*$, of
$\td{\Delta}_i^*$.

Let $\mathbf{\bf{\td{\Delta}^*}} =
\{\td{\Delta}_1^*,\ldots,\td{\Delta}_n^*\}$ be the tetrahedra
obtained from the collection of truncated-tetrahedra
$\{\overline{\Delta}_1,\ldots,\overline{\Delta}_n\}$ following the
crushing of the normal triangles in the surface $S$ to points. It
follows that there is a family $\mathbf{\Phi}^*$ of face-pairings
induced on the collection of tetrahedra
$\mathbf{\bf{\td{\Delta}^*}}$ by the face-pairings of $\C_X$ (coming
from the face-pairings of $\T$) as follows:
\begin{itemize}\item[-] if the face $\overline{\sigma}_i$ of
$\overline{\Delta}_i$ is paired with the face $\overline{\sigma}_j$
of $\overline{\Delta}_j$, then this pairing induces the pairing of
the face $\td{\sigma}_i^*$ of $\td{\Delta}_i^*$  with the face
$\td{\sigma}_j^*$ of $\td{\Delta}_j^*$ ;\item [-] if the face
$\overline{\sigma}_i$ of $\overline{\Delta}_i$ is paired with a face
of a truncated-prism in a chain of truncated-prisms and the face
$\overline{\sigma}_j$ of the truncated-tetrahedron
$\overline{\Delta}_j$ is also paired with a face of this chain of
truncated-prisms, then the face $\td{\sigma}_i^*$ of
$\td{\Delta}_i^*$ has an induced pairing with the face
$\td{\sigma}_j^*$ of $\td{\Delta}_j^*$ through the chain of
truncated-prisms.\end{itemize}

Hence, we get a $3$--complex
$\boldsymbol{\td{\Delta}}^*/\boldsymbol{\td{\Phi}}^*$, which is a
$3$--manifold except, possibly, at its vertices. We will denote the
associated ideal triangulation by $\T^*$. We call $\T^*$ the ideal
triangulation obtained by {\it crushing the triangulation $\T$ along
$S$}. We denote the image of a tetrahedron $\td{\Delta}^*_i$ by
$\Delta^*_i$ and, as above, call  $\td{\Delta}^*_i$ the lift of
$\Delta^*_i$.

We have the following  version of the Fundamental Theorem for Crushing Triangulations along a  Normal Surface.  A more general version and its proof appear as Theorem 4.1 in \cite{jac-rub-0eff}.  

\begin{thm}\label{combinatorial-crush} Suppose $\T$ is a triangulation of a compact, orientable $3$--manifold
or an ideal triangulation of the interior of a compact, orientable
$3$--manifold $M$. Suppose $S$ is a closed normal surface embedded in
$M$, $X$ is the closure of a component of the complement of $S$,
and $X$ does not contain any vertices of $\T$. If
\begin{enumerate}
\item[i)] $X\ne \bbb{P}(\C_X)$ \item[ii)] $\bbb{P}(\C_X)$ is a trivial
product region for $X$,  and \item[iii)] there are no cycles of
truncated prisms in $X$,  \end{enumerate}
then the triangulation $\T$ can be crushed along $S$ and the ideal triangulation $\T^*$ obtained by crushing $\T$ along $S$ is
an ideal triangulation of $\open{X}$.
\end{thm}

In this situation, we say the triangulation $\T$ admits a \emph{combinatorial crushing along $S$}. Notice that in the case of a combinatorial crushing, the tetrahedra in the ideal triangulation $\T^*$ are in one-one correspondence with the truncated tetrahedra in the cell decomposition $\C_X$ of $X$.  The latter collection of tetrahedra comes from truncating a sub collection of the tetrahedra of $\T$ and can be thought of as actually being a sub collection of the tetrahedra of $\T$; the face identifications for $\T^*$ are induced by the face identifications of $\T$.

\subsection{Inflating ideal triangulations}

A triangulation $\T$ of the compact $3$--manifold $M$ is said to be
a {\it minimal-vertex triangulation} if for any other triangulation
$\T_1$ of $M$ the number of vertices of $\T$ is no more than the
number of vertices of $\T_1$, $\abs{\T^{(0)}}\le\abs{\T_1^{(0)}}$.
If $M$ is closed, then $M$ has a one-vertex triangulation; hence, a minimal-vertex triangulation of $M$ is a
one-vertex  triangulation \cite{jac-rub-layered}. If $M$ is a
compact $3$--manifold with boundary, no component of which is a
$2$--sphere, then $M$ has a triangulation with all of its
vertices in the boundary and then just one vertex in each boundary
component (\cite{jac-rub-0eff}; hence, for such a manifold a minimal-vertex triangulation has all the vertices in the boundary and then just one vertex in each boundary component. These are the triangulations we are interested in and rather than write all of this out, we just say minimal-vertex triangulation. The proofs for these conclusions on minimal-vertex triangulations used here can be found in \cite{jac-rub-0eff} and are relevant to algorithms as they use crushing triangulations along normal surfaces, and result in fewer tetrahedra than the given triangulation. However, the layered triangulations for closed 3-manifolds, given in \cite{jac-rub-layered}, are one-vertex triangulations, one can easily show that any triangulation of a 3-manifold with boundary can be modified to one with all vertices in the boundary, then a standard "close-the-book" method reducing a triangulation of a closed surface not equal to $S^2$ to a one-vertex, minimal triangulation gives the general result for all 3-manifolds. 

 If $M$ is a compact 3--manifold with boundary and $\T$ is a triangulation of $M$ with normal boundary, then if $\T$ admits a crushing along the boundary-linking normal surface, we say $\T$ can be \emph{crushed along $\bdy M$}. 
 
 If $S$ is a triangulated surface, we say that a subcomplex $\xi$ in the 1--skeleton of the triangulation of $S$ is a \emph{frame} in $S$ if $\xi$ is a spine for $S$ (its complement is a connected open cell in $S$)  and is a minimum among spines, with respect to set inclusion.  In any triangulation of $S$ there are many choices for a frame.  See Figure \ref{f-frames} for examples of frames in a torus.   A vertex in a frame is called a \emph{branch point}  if it has index greater than two.  The closure of a component of a frame minus its branch points is called a \emph{branch}.  For the examples in Figure \ref{f-frames}, that on the left has one branch point of index 4 and two branches while that on the right has two branch points, each of index 3, and three branches.

\begin{figure}[htbp]

            \psfrag{L}{$\xi$}
        \vspace{0 in}
        \begin{center}
  \includegraphics[width=2.75 in]{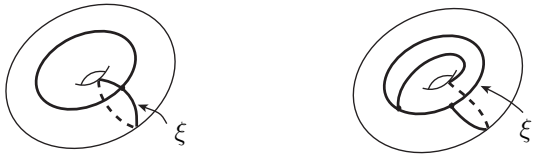}
\caption{There are 
only two possible topological types for frames in a triangulation of the torus.} \label{f-frames}
\end{center}
\end{figure}

 \begin{defn} If $\T^*$ is an ideal triangulation of $\open{X}$, the interior of the compact 3--manifold $X$, an \emph{inflation of $\T^*$} is a minimal-vertex triangulation $\T$ of $X$ with normal boundary that admits a \underline{combinatorial crushing} along $\bdy X$ for which the ideal triangulation obtained by crushing $\T$ along $\bdy X$ is the ideal triangulation $\T^*$ of $\open{X}$. \end{defn}

A construction for inflations of ideal triangulations of 3--manifolds is developed  in \cite{jac-rub-inflate}.  We discuss the construction here but reference the reader to \cite{jac-rub-inflate} for complete details. For a given ideal triangulation of the interior of a compact 3--manifold with boundary, there is not a unique inflation; however, all inflations of a given ideal triangulation share many common properties, some of which play a crucial role in this work.

Our construction begins with the choice of a ``frame" in the 1--skeleton of the induced triangulation of each vertex-linking surface of $\T^*$.  If $v^*$ is an ideal vertex of $\T^*$, we will use the notation $S_{v^*}$ for the vertex-linking surface of $v^*$ and $\xi$ for a frame in $S_{v^*}$. If there are a number of ideal vertices,  then for an ideal vertex $v_i^*$ we use $S_{v_i^*}$ for the vertex-linking surface and $\xi_i$ for a frame in $S_{v_i^*}$. We let $\Lambda = \xi_1\cup\xi_2\cup\cdots\cup\xi_k$ denote the union of the frames from all the vertex-linking surfaces.

 An inflation of an ideal triangulation $\T^*$ of $\open{X}$ includes all the tetrahedra of $\T^*$ and then, guided by the frame $\Lambda$, new tetrahedra are added to the tetrahedra of $\T^*$ and new  face identifications  are determined (discarding some of the face identifications of $\T^*$, using some of the face identifications of $\T^*$, and adding some new face identifications)  to arrive at a minimal-vertex triangulation $\T_\Lambda$ of $X$.  The triangulation $\T_\Lambda$ will have normal boundary that admits a combinatorial crushing of $\T_\Lambda$ along  $\bdy X$, crushing $\T_\Lambda$ back to the ideal triangulation  $\T^*$.  Figure \ref{f-inflate-scheme-lite-ann} provides a schematic for going between an ideal triangulation $\T^*$ of $\open{X}$ and a normal boundary, minimal-vertex triangulation $\T_\Lambda$ of $X$.

 \begin{figure}[htbp]

            \psfrag{X}{$(X,\T_\Lambda)$}\psfrag{Y}{$(\open{X},\T^*)$}
            \psfrag{1}{\tiny {$v_1^*$}} \psfrag{2}{\tiny {$v_2^*$}} \psfrag{3}{\tiny {$v_3^*$}}
            \psfrag{a}{\tiny {$\xi_1$}} \psfrag{b}{\tiny {$\xi_2$}} \psfrag{c}{\tiny {$\xi_3$}}
            \psfrag{x}{\tiny {$B_{\tiny {v_1}}$}} \psfrag{y}{\tiny {$B_{\tiny {v_2}}$}}\psfrag{z}{\tiny {$B_{\tiny {v_3}}$}}
            \psfrag{A}{\scriptsize {$S_{\tiny {v_1^*}}$}}\psfrag{B}{\scriptsize {$S_{\tiny {v_2^*}}$}}
            \psfrag{C}{\scriptsize {$S_{\tiny {v_3^*}}$}}
            \psfrag{u}{\scriptsize {$S_{\tiny {v_1}}$}}\psfrag{v}{\scriptsize {$S_{\tiny {v_2}}$}}
            \psfrag{w}{\scriptsize {$S_{\tiny {v_3}}$}}
            \psfrag{m}{\scriptsize
            {\bf Inflate}} \psfrag{n}{\scriptsize
            {\bf Crush}}
            \psfrag{L}{$\Lambda =\xi_1\cup\xi_2\cup\xi_3$}

        \vspace{0 in}
        \begin{center}
          \includegraphics[width=4.5 in]{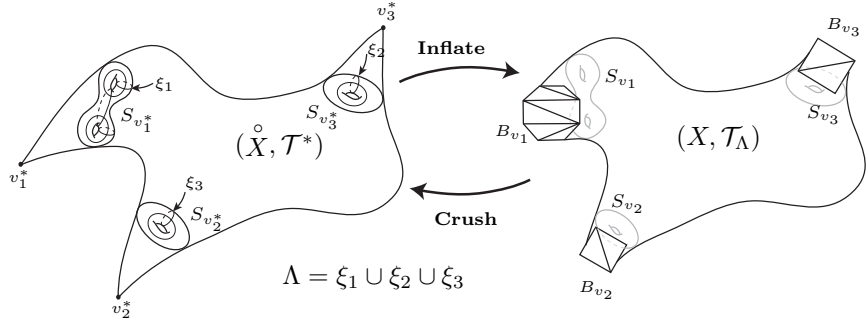}
        \caption{ Schematic of an inflation of an ideal triangulation using the collection of frames in $\Lambda =\xi_1\cup\xi_2\cup\xi_3$}
        \label{f-inflate-scheme-lite-ann}
\end{center}
\end{figure}

 An ideal vertex $v^*$ in $\T^*$ inflates to a minimal (one-vertex) triangulation of a component $B_v$ of $\bdy X$, which is induced by $\T_\Lambda$.   The vertex-linking surface $S_{v^*}$ about the ideal vertex $v^*$  inflates to a normal surface $S_v$ in the triangulation $\T_\Lambda$, which is boundary-linking $B_v$.

\subsection{Closed normal surfaces} In this section, we show there is a
bijective correspondence  between the closed normal surfaces
in an ideal triangulation $\T^*$ and the closed normal surfaces in
any inflation $\T$  of $\T^*$.  We provide the details for this in Theorem \ref{bijection-ideal-inflate} of
this section. 
A special case of this relationship
is a key ingredient for the work in this paper and relates  boundary
parallel normal surfaces in an inflation of an ideal triangulation
with normal surfaces parallel to vertex-linking surfaces in the
ideal triangulation. 

\begin{lem} Suppose $M$ is a compact $3$--manifold with nonempty
boundary and $\T$ is a triangulation of $M$ with normal boundary. An
embedded normal surface in $\T$ that contains all the quad types of
a boundary-linking surface has that boundary-linking surface as a
component.\end{lem}

\begin{proof} Suppose $S$ is an embedded normal surface in $\T$, $\td{B}$
is a boundary-linking surface, and all quad types of $\td{B}$ are
represented as quad types in $S$. Then $S$ and $\td{B}$ are
contained in the carrier of $S$, a face of compatible (no two
distinct quad types in the same tetrahedron) normal solutions in the
solution space of embedded normal surfaces. It may be the case that
 $\td{B}$ is in a proper
face. Since, $\td{B}$ has no more quad types that $S$, it follows
that there is a normal surface $R$ and positive integers $k,n,$ and
$m$ so that $kS = nR + m\td{B}$. However, we can move $\td{B}$ by a
normal isotopy so that it does not meet $R$. Hence, we have $\td{B}$
a component of $kS$ and therefore a component of $S$.\end{proof}

\begin{lem}\label{crush normal to normal} Suppose $M$ is a compact $3$--manifold with nonempty boundary, no component of which is a 2--sphere. Suppose
 $\T^*$ is an ideal triangulation of
$\open{M}$, and $\T$ is an inflation of $\T^*$.  The combinatorial crushing map determined by crushing $\T$ along $\bdy M$ takes a closed normal
surface $S$ in $\T$  to a closed normal surface $S^*$ in
$\T^*$; furthermore, $S$ and $S^*$ are homeomorphic.
\end{lem}

\begin{proof}  Let $X$ denote the component of the complement of the
boundary-linking surfaces that does not meet $\bdy M$. Then $X$
contains none of the vertices of $\T$ and has a nice
cell-decomposition $\mathcal{C}$; furthermore, this
cell-decomposition combinatorially crushes along the boundary-linking surfaces to
the ideal triangulation $\T^*$.

Let $S$ be a closed normal surface in $\T$. The surface $S$ has an induced cell-decomposition from $\T$
consisting of normal quadrilaterals and normal triangles. Since $S$
is a closed normal surface, we may assume $S$ does not meet any of
the boundary-linking surfaces along which we are crushing, and thus
$S\subset X$.

If a normal quad or normal triangle of $S$ is in a
truncated-tetrahedron in $\C$, then upon crushing, the
truncated-tetrahedron is taken to a tetrahedron of $\T^*$ and the
normal cells of $S$ in the truncated-tetrahedron are carried
isomorphically onto normal cells in $\T^*$ (see Figure
\ref{f-crush-normal} A). If a normal quad or normal triangle of $S$
is in a truncated-prism of $\C$, then the truncated-prism is crushed
to a face in $\T^*$ and the normal cells of $S$ are crushed to
normal arcs in that face. The normal arcs in the hexagonal faces of
the truncated-prisms correspond to where $S$ meets these hexagonal faces and are matched under the crushing map from the various truncated prisms in a chain of truncated prisms. Arcs in
the trapezoidal faces of the truncated-prism crush to points in the edges of the face in which the truncated prism crushes (see
Figure \ref{f-crush-normal} B). Finally, the normal cells of $S$ in
the product blocks of $\C$ are ``horizontal" triangles in the
triangular product blocks and ``horizontal" quadrilaterals in the
quadrilateral blocks and, hence, each is crushed to a single point
in an edge of $\T^*$ (see
Figure \ref{f-crush-normal} C). The crushing in the trapezoidal faces of the
truncated-prisms and the product blocks are consistent. It follows
that the image of $S$ is formed from the collection of normal
triangles and normal quadrilaterals of $S$ that are in the
truncated-tetrahedron of $\C$ by identifications along their edges
and gives a normal surface $S^*$ in $\T^*$.

 \begin{figure}[htbp]

        \vspace{0 in}
        \begin{center}
        \epsfxsize=3.5 in
        \epsfbox{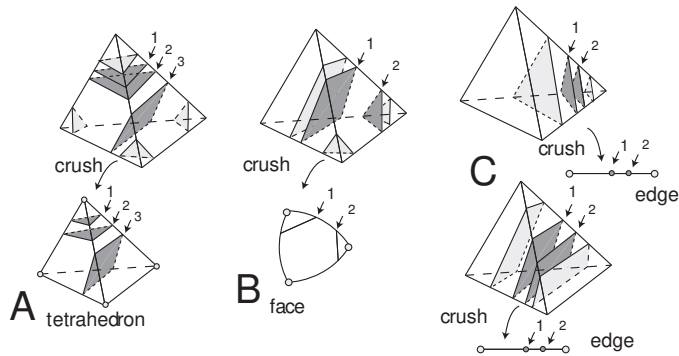}\caption{A.Normal disks in truncated tetrahedra go to normal disks. B.Normal disks in truncated prisms go to normal arcs. C.Normal disks in product blocks go to points. }
        \label{f-crush-normal}
\end{center}
\end{figure}

To see that $S$ and $S^*$ are homeomorphic, we observe that the
inverse image of a point in the interior of a normal quad or normal
triangle in $S^*$ is a point in the interior of a normal quad or
normal triangle in $S$. The inverse image of a point in an edge of
$S^*$ is either a point in an edge of $S$ or a sequence of arcs in normal quads or normal triangles of $S$; there are no
cycles of truncated prisms and so no cycles of cells of $S$ in
truncated-prisms. The inverse image of a vertex of $S^*$ is a horizontal cross section $K_i\times {t}$ in one of the component product pieces $\mathbb{P}_i = K_i\times I$ of the combinatorial product $\mathbb{P}(\C_X)$. Hence, $K_i$ is a contractible planar complex.  Thus for each point of $\S^*$ its inverse image in $S$ is a contractible planar complex and so the combinatorial crushing map gives a cell-like map from $S$ to $S^*$ and by a 2-dimensional versions of \cite{arm, sie} it follows that $S$ and $S^*$ are homeomorphic.
\end{proof}

\begin{thm} \label{bijection-ideal-inflate} Suppose $M$ is a compact 3--manifold with nonempty boundary no component of which is a 2--sphere. Suppose
 $\T^*$ is an ideal triangulation of
$\open{M}$, and $\T$ is an inflation of $\T^*$.  The combinatorial crushing map determined by crushing $\T$ along $\bdy M$ induces a bijection between the closed normal
surfaces  in $\T$  and the closed normal surface in
$\T^*$; furthermore, corresponding surfaces are homeomorphic.\end{thm}
\begin{proof} By Lemma \ref{crush normal to normal} we only need to show that the combinatorial crushing induces a bijection between the closed normal surfaces of $\T$ and those of $\T^*$.  

First we shall show that the correspondence is injective. Suppose $S_1$ and $S_2$ are distinct closed normal surfaces in $\T$. Since both are closed, we may assume that (up to normal isotopy) they do not meet any boundary-linking surface in $\T$.  Let $X$ denote the component of the complement of the boundary-linking normal surfaces in $\T$, which does not meet $\bdy M$ and let $\C_X$ denote the nice cell decomposition on $X$ induced by $\T$.   Then $S_1$ and $S_2$ are distinct normal surfaces in $\C_X$  and hence, have distinct normal coordinates. By Lemma \ref{crush normal to normal} the combinatorial crushing of $\T$ to $\T^*$ takes $S_1$ and $S_2$ to closed normal surfaces $S_1^*$ and $S_2^*$ in $\T^*$, respectively.   We will show that $S_1^*$ and $ S_2^*$ have distinct normal coordinates.

If $S_1$ and $S_2$ have distinct sets of normal disks in a truncated tetrahedron of $\C_X$, then $S_1^*$ and $S_2^*$ have distinct normal disks in a tetrahedron of $\T^*$ and hence, $S_1^* \ne S_2^*$. 

If $S_1$ and $S_2$ have distinct normal disks in a truncated prism, say $\pi$, then they have distinct sets of normal arcs in a hexagonal face of $\pi$, which extend to distinct sets of normal arcs on all the hexagonal faces of the truncated prisms in the chain of truncated prisms containing $\pi$, which leads to a distinct set of normal disks in a truncated tetrahedron in which the chain terminates. This again  gives that $S_1^* \ne S_2^*$.  Finally if $S_1$ and $S_2$ have a distinct number of quads or triangles in a quadrilateral or triangular block, respectively, then $S_1$ and $S_2$ meet an entire product component, $K_i\times I$ in a distinct number of horizontal slices. The vertical frontier of a product $K_i\times I$ is made up of trapezoidal faces which are paired with trapezoidal faces of truncated prisms. Thus we have that $S_1$ and $S_2$ must meet a truncated prism in distinct normal disks. From the previous consideration, we have that $S_1^*$ and $S_2^*$ are distinct.  So, the correspondence is injective.

Now, we must show the correspondence is surjective. Suppose $S^*$ is a closed normal surface in $\T^*$.  First, we consider how, $S^*$ meets a tetrahedron of $\T^*$. Each tetrahedron of $\T^*$ is the image of a single truncated tetrahedron of $\C_X$ under the crushing map; hence, there is a unique choice of normal cells in these truncated tetrahedra of $\C_X$ (tetrahedra of $\T$) mapping to the normal cells of $S^*$. If $\alpha^*$ is a face of a tetrahedron of $\T^*$ and $\alpha^*$ meets $S^*$, then the inverse image of $\alpha^*$ is either a single face between two truncated tetrahedra in $\C_X$ or is the image of a chain of truncated prisms in $\C_X$ between two truncated tetrahedra in $\C_X$. If the inverse image of $\alpha^*$ is a single face matching two truncated tetrahedra, then there are well determined normal cells in each of these truncated tetrahedra determined by the normal cells in $S^*$. If there is a chain of truncated prisms determined by $\alpha^*$, then of the three possible families of normal arcs in $\alpha^*$, only one of the families determines quadrilaterals in any one of the truncated prisms in the chain determined by $\alpha$.  This again determines a unique way to fill in normal disks extending the normal disks in the truncated tetrahedra. Finally, for each product $K_i\times I$ there is a unique number of horizontal slices determined to complete a normal surface $S$ in $\T$ that crushes to $S^*$.
\end{proof}

Recall that any  ideal triangulation of the interior of a compact 3--manifold M with boundary, no component of which is a 2--sphere, has numerous inflations. By the previous theorem, all of these inflations have isomorphic sets of closed normal surfaces, which are, in turn,  isomorphic with the closed normal surfaces of the given ideal triangulation. Here we use isomorphic to mean a bijection between the sets of normal surfaces where corresponding surfaces are homeomorphic.

\begin{cor}\label{vertex-link is bdry-link} Suppose $M\ne \mathbb{B}^3$ is a compact, irreducible and
$\bdy$-irreducible $3$--manifold with nonempty boundary. Suppose
 $\T^*$ is an ideal triangulation of
$\open{M}$, and $\T$ is an inflation of $\T^*$.  There is a closed normal
surface in $\T$  isotopic into  $\bdy M$ but not normally isotopic into $\bdy M$ if and only if there is
 a closed normal surface in $\T^*$ that
is isotopic into a vertex-linking surface but is not normally isotopic into a  vertex-linking surface.\end{cor}

\begin{proof}  Suppose $S$ and $S^*$ are closed normal surfaces in $\T$ and $\T^*$, respectively, that correspond under the combinatorial crushing map taking $\T$ to $\T^*$. Then the closure of the components of the complement of $S$ have a correspondence under the combinatorial crushing map and corresponding components are homeomorphic.  Also, we have that the correspondence under the combinatorial crushing map  takes boundary-linking normal surfaces in $\T$ to vertex-linking normal surfaces in $\T^*$. Hence, if $S$ is isotopic into $\bdy M$, $S$ is isotopic to a boundary-linking surface and so, $S^*$ is isotopic to a vertex-linking surface in $\T^*$. The converse also follows.
\end{proof}

\section{efficient triangulations}
In this section we define and study annular-efficient triangulations and boundary-efficient triangulations ($\partial$-efficient  triangulations)
 of bounded 3-manifolds.  These results build on the study of $0$-efficient triangulations in \cite{jac-rub-0eff} and  \cite{jac-rub-layered} .

If $M$ is a $3$--manifold and $\T$ is a triangulation of $M$,  we say the
triangulation $\T$  is {\it $0$-efficient} if
\begin{enumerate}\item[(i)] $M$ is closed and the only normal
$2$--spheres are vertex-linking; or \item[(ii)] $\bdy M\ne\emptyset$
and the only normal disks  are vertex-linking.\end{enumerate}

Similarly, if $\T^*$ is an ideal triangulation of $\open{M}$, the interior of a compact 3-manifold, $\T^*$ is {\it $0$-efficient} 
\begin{enumerate}

\item[(iii)]  if there are no normal 2-spheres in $\T^*$.\end{enumerate}

It is shown in  Proposition 5.1 and Proposition 5.15 of
\cite{jac-rub-0eff} that if $\T$ is a $0$-efficient triangulation
of the compact $3$-manifold $M$ or $\partial M\ne\emptyset$  and  $\T^*$ an ideal triangulation of $\open{M}$, then for
\begin{enumerate}\item[(i)] $M$ closed,  then $M \ne \rppp$, is irreducible,
 and $\T$ has only one vertex, or $M=S^3$ and $\T$ has precisely two
 vertices.
\item[(ii)] $\bdy M\ne\emptyset$, then $M$ is irreducible,
$\bdy$-irreducible, there are no normal $2$-spheres, all the
vertices are in $\bdy M$, and there is precisely one vertex in each
boundary component, or $M = \mathbb{B}^3$. \item[(iii)] $\bdy M\ne\emptyset$ and $\open{M}$ is irreducible.\end{enumerate}

The following two results are from \cite{jac-rub-0eff} and show that the necessary topological conditions coming from a 0-efficient triangulation for a 3--manifold are sufficient to algorithmically construct a $0$-efficient triangulation for such manifolds.

\begin{thm} [Theorem 5.5, \cite{jac-rub-0eff}] If $M$ is a closed, orientable, irreducible
$3$--manifold distinct from $\rppp$ and $L(3,1)$,  then there is an algorithm that will modify any triangulation of $M$  to a $0$--efficient
triangulation.\end{thm}

Note: While $\rppp$ does not admit a 0-efficient triangulation, $L(3,1)$ does have a 0-efficient triangulation; however, there are triangulations of $L(3,1)$ that cannot be modified by our algorithm to a $0$-efficient triangulation; in fact, these triangulations show up as obstructions to our algorithm and exhibit that the manifold we started with is either $\rppp$ or $L(3,1)$.

The same method will modify any ideal triangulation of the interior of a compact, irreducible 3-manifold to a $(0)$-efficient ideal triangulation.

\begin{thm}  [Theorem 5.17, \cite{jac-rub-0eff}]  If $M\ne \mathbb{B}^3$ is a compact, orientable, irreducible,
$\bdy$--irreducible $3$--manifold, with non-empty boundary, then there is an algorithm that will modify any triangulation of $M$  to a
$0$--efficient triangulation.\end{thm}

\begin{figure}[htbp] \psfrag{B}{\tiny
{$B\subset\bdy M$}}\psfrag{N}{{$e \in \bdy M$}}\psfrag{M}{{$e \in
\open{M}$}}
 \psfrag{b}{\tiny{$B'\subset\bdy M$}}\psfrag{e}{\footnotesize{e}}
 \psfrag{V}{\tiny{$v$}}\psfrag{v}{\tiny
 {$v'$}}\psfrag{D}{\footnotesize{$D$}}\psfrag{d}{\footnotesize{$D'$}}\psfrag{A}{\footnotesize{cycle
 of quads}}

        \begin{center}
  \includegraphics[width=3.5 in]{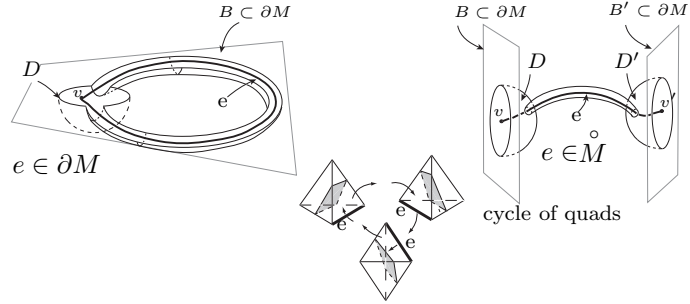}
\caption{Thin
edge-linking annuli. On the left is a thin edge-linking annulus
about the edge $e$ in the boundary. On the right, the edge $e$ is in
the interior of the manifold.} \label{fr-thin-annulus.eps}
\end{center}
\end{figure}

If $\T$ is a triangulation of the $3$--manifold $M$, we say a normal
annulus in $M$ is {\it thin edge-linking} if and only if it is
normally isotopic into an arbitrarily small regular neighborhood of an
edge in the triangulation. In Figure \ref{fr-thin-annulus.eps}, we
give the two possible examples of thin edge-linking, normal annuli; the figure on the left is
a thin edge-linking normal annulus for an edge $e$  in $\bdy M$ and the one
on the right is a thin edge-linking, normal annulus for an edge $e$ in
$M$ with vertices in distinct boundary components of $M$.  A necessary and sufficient condition for an edge $e$ in $\bdy M$ or
a properly embedded edge $e$ in $\T$ with boundary in distinct components of $\partial$M to have a thin edge-linking
annulus about it is that no face in the triangulation has two edges
identified to the edge $e$. A thin edge-linking, normal
annulus is not essential. F     or example, a thin edge-linking, normal annulus along an edge in the boundary is not $\bdy$-incompressible and if $M$ is incompressible and $\bdy$-incompressible is isotopic into $\bdy M$; one in the interior between distinct boundary components is not incompressible.  If $M$ is a compact $3$--manifold with nonempty boundary and $\T$ is
a triangulation of $M$, we say $\T$ is {\it annular-efficient} if
and only if $\T$ is $0$--efficient and the only normal, incompressible annuli are thin edge-linking annuli. David Bachman and
Saul Schleimer, who have independently studied similar conditions to our annular-efficient triangulations, use the term  $\sfrac{1}{2}$-${\it efficient}$ for  {\it annular-efficient}.

In the following, we have necessary topological conditions for a $3$-manifold with an annular-efficient triangulation.

\begin{prop} \label{annular-eff-properties}	Suppose  $M\ne \mathbb{B}^3$ is a compact $3$--manifold with boundary and has an
annular-efficient triangulation. Then $M$ is irreducible,
$\bdy$--irreducible, and an-annular. Furthermore, there are no normal
 $2$--spheres, all the vertices are in $\bdy M$, the only normal disks are vertex-linking, and there
is precisely one vertex in each component of $\bdy M$.\end{prop}
\begin{proof} Since an annular-efficient triangulation is
$0$--efficient, it follows from Theorem 5.15 of \cite{jac-rub-0eff}
that $M$ is irreducible and $\bdy$--irreducible and there are no
normal $2$--spheres, all the vertices are in $\bdy M$, and there is
precisely one vertex in each boundary component of $M$. Hence, it remains to prove that $M$ is an-annular. If
there is a properly embedded, essential annulus in $M$, then for any
triangulation, there must be a normal, embedded, essential annulus
in $M$. In particular, this would need to be the case for the given
annular-efficient triangulation. However, a normal, embedded,
essential annulus can not be thin edge-linking, as a thin
edge-linking annulus is either $\bdy$-compressible and (in an irreducible and $\bdy$-irreducible 3--manifold) parallel into the boundary of the manifold or is along an edge between distinct boundary components and is compressible. Therefore, a thin-edge-linking annulus cannot be essential. Thus the assumption of an
embedded, essential annulus leads to contradictions.\end{proof}

\begin{prop}\label{decide-annular-eff} Given a triangulation of a compact, orientable 3--manifold with nonempty boundary, no component of which is a 2--sphere, there is an algorithm to decide if the triangulation is annular--efficient. Furthermore, if the triangulation is not 0-efficient, the algorithm will construct a normal disk that is not vertex-linking; and if the triangulation is 0--efficient and is not annular-efficient,  the algorithm will construct an incompressible,  normal annulus that is not  thin edge-linking.\end{prop}

\begin{proof} By Proposition 5.19 of \cite{jac-rub-0eff}, it can be decided if the given triangulation is 0--efficient; and, if there is a normal disk that is not vertex-linking, the algorithm will construct one.  So, we may assume the only normal disks are vertex-linking; i.e., the triangulation is 0--efficient.

If there is an incompressible, normal annulus that is not thin edge-linking, then consider one,  say $A$, where the carrier of $A, \C(A)$, has minimal dimension.  If $\C(A)$ is not a vertex, then there are normal surfaces $X$ and $Y$ in proper faces of $\C(A)$ and positive integers $k, n$ and $m$ so that $kA = nX + mY$.  Since the triangulation is 0--efficient, the only positive Euler characteristic normal surfaces are vertex-linking normal disks; hence, neither $X$ nor $Y$ has a component with positive Euler characteristic.  It follows that the components of both $X$ and $Y$ have Euler characteristic zero.  Since $A$ has essential boundary, every component of $X$ and $Y$ with boundary has essential boundary. It follows that each component of $X$ and $Y$ with boundary is either an annulus or  a M\"obius band.  However, if a component of $X$ or $Y$ is a M\"obius band, then we would have a normal annulus that is not thin edge-linking and  carried by a proper face of $\C(A)$, which contradicts our choice of $A$.  So, any component of $X$ or $Y$ that has boundary is an annulus and, again by our choice of $A$, these annuli must be thin edge-linking.  Now, a thin edge-linking annulus can be normally isotoped to miss any closed normal surface. It follows that both $X$ and $Y$ must have components with boundary and as such both must have components that are thin edge-linking annuli. However, the Haken sum of two thin edge-linking annuli is either two thin edge-linking annuli (the two annuli are the same or have their boundaries in distinct boundaries of the 3--manifold), or has a component a vertex-linking disk.  Both possibilities lead to a contradiction that $A$ is connected and not a thin edge-linking annulus.  

It follows that if the triangulation is 0--efficient and there is a normal annulus that is not thin edge-linking, then there is one at a vertex of the projective solution space for the triangulation. Furthermore, we can recognize if a normal surface is a thin edge-linking annulus. \end{proof}

Our method to understand boundary slopes of normal  surfaces (hence, boundary slopes of properly embedded incompressible and $\bdy$-incompressible surfaces) is to show that for manifolds having the topological conditions necessary for a manifold with an annular-efficient triangulation are topological conditions sufficient to modify any given triangulation of that manifold to an annular-efficient triangulation. In fact, we introduce a very useful new notion for efficient triangulations of bounded 3-manifolds, as well as for ideal triangulations of their interiors; it is boundary-efficient triangulations. 

A triangulation $\T$ of the bounded 3-manifold $M \ne \mathbb{B}^3$ is said to be {\it boundary-efficient} iff the only normal surface isotopic into a component of $\bdy M$ is a boundary-linking normal surface. If we have an ideal triangulation $\T^*$ of the interior of the compact 3--manifold with boundary, we say $\T^*$ is {\it boundary-efficient} iff the only normal surface in $\T^*$ isotopic to a vertex-linking surface is the vertex-linking surface itself.

\begin{prop}\label{bdy-eff-gives-annular-eff}  Suppose $M$ is a compact, orientable 3--manifold with boundary and $M$ is irreducible,  $\bdy$--irreducible, and an-annular. If  $\T$ is a $\bdy$-efficient triangulation of $M$, then $\T$ is annular-efficient. \end{prop}

\begin {proof} If there is a normal disk $D$ in $\T$ that is not vertex-linking or a
normal annulus $A$ with incompressible boundary in $\T$ that is not thin
edge-linking, then by a barrier surface argument using $\bdy M\cup D$ or $\bdy M\cup A$, respectively, as barriers, there is a closed
normal surface in $\T$ that is isotopic into $\bdy M$ but is not
\underline{normally} isotopic into $\bdy M$.  Hence, $\T$ must be both 0-efficient and have every incompressible annulus thin edge-llinking.\end{proof}

It follows that for a $\bdy$-efficient triangulation $\T$  of an irreducible, $\bdy$-irreducible  3-manifold, the triangulation $\T$ is necessarily 0-efficient and if, in addition, the manifold has no essential annuli  (an-annular), the triangulation is annular-efficient.  In general, a $\bdy$-efficient triangulation does not imply annular-efficient nor does an annular-efficient triangulation imply $\bdy$-efficient.

The following theorem gives conditions under which we can decide if a triangulation of a compact 3-manifold with boundary has a closed, normal surface that is isotopic into the boundary but is not normally isotopic into the boundary.  Similarly, for an ideal triangulation of the interior of a compact 3--manifold with boundary, it can be determined if there is a normal surface other than the vertex-linking surface that is isotopic to the vertex-linking surface,  Corollary \ref{determine-vertex-linking}.

\begin{thm} \label{determine-bdry-eff} Suppose $M$ is a compact, orientable 3--manifold with boundary and $M$ is irreducible,  $\bdy$--irreducible, and an-annular.  Then for any triangulation $\T$ of $M$ there is an algorithm to decide if there is a closed normal surface that is isotopic into $\bdy M$ but is not normally isotopic into $\bdy M$. Furthermore, if there is one the algorithm will construct one. \end{thm}

\begin{proof}  If $S$ and $S'$ are disjoint normal surfaces embedded in $M$ and both are isotopic into $\bdy M$, we say $S'$ is \emph{larger-than} $S$ if $S$ is normally isotopic into the product region between $S'$ and $\bdy M$. Being larger-than is a partial order on closed normal surfaces embedded in $M$.  

Suppose there is a normal surface in $M$ that is isotopic into $\bdy M$ but is \underline{not} normally isotopic   into $\bdy M$.  By Kneser's Finiteness Theorem \cite{kne} there are maximal (relative to the preceding partial order) such surfaces.  Suppose $S$ is a maximal normal surface that is isotopic into $\bdy M$ but not normally isotopic into 
$\bdy M$. We claim $S$ is a fundamental surface.

Suppose $S$ is not fundamental. Then $S = X + Y$ is a nontrivial Haken sum.  Hence, there are exchange annuli between $X$ and $Y$.   Suppose $A$ is an exchange annulus. Then $A$ is a 0-weight annulus meeting $S$ only in its boundary.  There are two possibilities: either $A$ is not in the product region between $S$ and $\bdy M$ or $A$ is in the product region between $S$ and $\bdy M$.  Since $S$ is isotopic into $\bdy M$ and $M$ is an-annular, then for either possibility, $A$ is isotopic into $S$.   

Let $N=N(S\cup A)$ be a small regular neighborhood of $S\cup A$, then  $N$ has three boundary components; one is a torus bounding a solid torus, which is a product between $A$ and an annulus $A'$ in $S$, another is a surface normally isotopic to $S$, and the third is a surface isotopic to $S$ but possibly not normal and even if normal is not  normally isotopic to $S$.  The complex $S\cup A$ is a barrier (see \cite{jac-rub-0eff}) and thus each boundary component of $N$ can be normalized in the closure of the component of its complement not meeting $S\cup A$. 

Suppose $A$ is not in the closure of the product region between $S$ and $\bdy M$.  Then
the component of $\partial N$ isotopic to $S$ can be normalized missing $S\cup A$ to a normal surface $S'$.  Since $M$ is irreducible and $\bdy$--irreducible, $S'$ is isotopic to $S$ and therefore, isotopic into $\bdy M$.  Moreover $S'$ is not normally isotopic to $S$ or normally isotopic into $\bdy M$ due to the annulus $A$.  But $S'$ is larger than $S$, which contradicts $S$ being maximal.

Suppose $A$ is in the closure of the product region between $S$ and  $\bdy M$.  Then $A$ co-bounds a solid torus which is a product between $A$ and an annulus $A'$ in $S$.  We observe that $X \neq Y$, for if this were not the case, then $X$ (and $Y$) would be one-sided and $M$ would be a twisted I-bundle, contradicting $M$ being an-annular.  Hence, there must be a trace curve in $A'$. Suppose we have selected $A'$ in this situation so that it has a minimal number of trace curves.  Since there is a trace curve in $A'$, there is another exchange annulus $A_1$ for $S$ meeting $A'$ in at least one of its boundary components.  If $A_1$ is 
not in the closure of the product region between $S$ and $\bdy M$, then the preceding argument gives a contradiction to our selection of $S$.  So we may assume $A_1$  is, like $A$,  in the closure of the product region between $S$ and  $\bdy M$ and therefore in the solid torus co-bounded by  $A$ and $A'$.  It follows that $A_1$ 
co-bounds a solid torus which is a product between $A_1$ and an annulus $A_1'$ in $A'\subset S$.  However, then $A_1'$  has fewer trace curves than $A'$ contradicting our choice of the exchange annulus $A$. 

So, there is a closed normal surface that is isotopic into $\bdy M$ and not normally isotopic into $\bdy M$ if and only if there is such a surface among the fundamental surfaces for the triangulation $\T$ of $M$.  By \cite{jac-tol} given any normal surface we can determine if it is isotopic into $\bdy M$ and it is straight forward to recognize if it is normally isotopic into $\bdy M$. It follows if such a surface exists, we can construct one.  \end{proof}

The argument carries over to an analogous result in the case of an ideal triangulation.

\begin{thm}\label{determine-vertex-linking} Suppose $M \ne \mathbb{B}^3$ is a compact, orientable,  irreducible,  $\bdy$--irreducible, and an-annular 3-manifold with nonempty boundary.  Then for any ideal triangulation $\T^*$ of $\open{M}$ and any ideal vertex $v^*$ of $\T^*$, there is an algorithm to decide if there is a closed normal surface that is isotopic into the vertex-linking surface of $v^*$ but is not normally isotopic into the vertex-linking surface of $v^*$. Furthermore, if there is one, the algorithm will construct one. \end{thm}

Recall, if 
we are given a compact, orientable $3$--manifold $M$ with boundary via a triangulation $\T$, algorithms
exist to determine if it is irreducible \cite{sch, jac-tol, rubin-poly, tho},
$\bdy$--irreducible \cite{haken-norflach, jac-tol},
or an-annular \cite{haken-homo, jac-sed-dehn}; from this and the previous two theorems, we have

\begin{cor} Suppose  $M$  a compact, orientable 3-manifold with boundary.  it can be decided if a triangulation of $M$ or an ideal triangulation of $\open{M}$ is $\bdy$-efficient. \end{cor}

We now continue our program of constructing, under various topological hypotheses, essential triangulations from arbitrary triangulations. We need a couple of preliminary results.   We refer the reader to  \cite{jac-rub-0eff}  for crushing a triangulation of a compact, orientable 3-manifold with boundary $X$ along a normal surface to an ideal triangulation of $\open{X}$ and to \cite{jac-rub-inflate} for the inverse operation of inflating an ideal triangulation of $\open{X}$ to a minimal-vertex, normal-boundary triangulation of $X$. Both methods are reviewed in Section \ref{crushing-inflations} of this work.\\

\begin{thm}\label{thm:bdry-eff} Suppose $M\ne \mathbb{B}^3$ is a compact, orientable,
irreducible, $\bdy$--irreducible,  and an-annular $3$--manifold with
nonempty boundary. Then there is an algorithm that will modify any triangulation $\T$ of $M$   to
a boundary-efficient triangulation of ${M}$ and, hence, an annular-efficient triangulation of $M$.
\end{thm}

Before providing the proof of Theorem \ref{thm:bdry-eff}, we have a result from \cite{jac-rub-0eff} that provides, under our hypothesis and for any triangulation $\T$ of $M$, a nice algorithm to modify $\T$ to an ideal triangulation $\T'$  of $\open {M}$.\\

\noindent Theorem 7.1 \cite{jac-rub-0eff} Suppose $M$ is a compact, irreducible, $\bdy$-irreducible, an-annular 3-manifold, then any triangulation $\T$  of $M$ admits a crushing along a closed normal surface, each component of which is isotopic to a distinct component of $\bdy M$, giving an ideal triangulation $\T'$ of $\open{M}$; furthermore,  $\abs{\T'}<\abs{\T}$.\\

This exhibits immediately one of the benefits of our approach. The algorithm constructs an ideal triangulation  $\T'$ of $\open M$ from the tetrahedra and face identifications of the given triangulation $\T$ of $M$. Furthermore, the crushing offers substantial reduction of potential complexity issues, since  $\abs{\T'}<\abs{\T}$. Note that for this 
step the crushing may not be a combinatorial crushing (see Section \ref{crushing-inflations}) and may require an induced product region, which is handled in the proof of Theorem 7.1 in \cite{jac-rub-0eff}. 

We remark here and discuss later in this paper related work by M. Lackenby \cite{lack-h-genus} that depends on results of S. Matveev \cite{matveev-book} for the existence of an ideal triangulation  and presents a discussion of an algorithm to construct an ideal triangulation of $\open{M}$ having a partially flat angled ideal structure. The conditions on the manifolds in \cite{lack-h-genus} include our conditions for the proof of Theorem 7.1 of  \cite{jac-rub-0eff} but add the additional condition that there are no essential  (incompressible and not isotopic into boundary) tori. This latter condition enables Lackenby to get the partially flat angled ideal triangulation.  

Next, we have under the same hypothesis on $M$ as that of Theorem 7.1 \cite{jac-rub-0eff}, any ideal triangulation of $\open {M}$ can be modified to a $\bdy$-efficient triangulation of $\open{M}$.

\begin{thm}\label{thm:ideal-bdry-eff} Suppose $M\ne \mathbb{B}^3$ is a compact, orientable,
irreducible, $\bdy$--irreducible,  and an-annular $3$--manifold with
 boundary. Then there is an algorithm that will modify any ideal triangulation of $\open {M}$  to
a boundary-efficient triangulation of $\open {M}$.
\end{thm}

\begin{proof}.
 We are given an ideal triangulation $\T' $ of $\open{M}$.  By Corollary \ref{determine-vertex-linking} we can decide if there is a closed normal surface in $\T'$ that is isotopic to a vertex-linking surface but is not itself a vertex-linking surface. If there is not one, then $\T'$ is a $\bdy$-efficient triangulations and we are done.  On the other hand, if there is one, then the algorithm will  construct one, say $F$, and $F$ is isotopic to a vertex-linking surface  but is not itself vertex-linking.  We wish to crush the triangulation $\T'$ along $F$.  However, to keep the situation consistent with the cell decompositions we like and our methods, if we have ideal vertex-linking surfaces  $S_{v_1^*},\ldots, S_{v_n^*}$ and notation has been chosen so that $F$ is isotopic to $S_{v_1^*}$ but is not normally isotopic to $S_{v_1^*}$, then we wish to crush the triangulation along the collection of surfaces $F,  S_{v_2^*},\ldots,  S_{v_n^*}$. 

If $X$ is the closure of the component of the complement of $F,  S_{v_2^*},\ldots,  S_{v_n^*}$ not meeting any of the ideal vertices of $\T'$, then $X$ is homeomorphic to $M$ and we can proceed in finding a collection of normal surfaces along which to crush the triangulation $\T'$, replacing the collection $E_1,\ldots, E_n$ in the proof of Theorem 7.1 of \cite{jac-rub-0eff} with the collection $F,  S_{v_2^*},\ldots,  S_{v_n^*}$.  Hence, we arrive at an ideal triangulation $\T^*$ of $\open{X}$ (homeomorphic with $\open{M}$) obtained by crushing $\T'$ along a collection of normal surfaces with at least one of them not vertex-linking.  

The tetrahedra of $\T^*$ come from a subset of the tetrahedra of $\T'$ that become truncated tetrahedra in the cell decomposition $\C_X$ of $X$.  Now, since one of the normal surfaces along which we are crushing is not vertex-linking, it must contain a normal quadrilateral, and hence, at least one of the tetrahedra of $\T'$ gives a truncated prism in $\C_X$ and we have that $\abs{\T^*} < \abs{T'}$. 

It follows that the process must stop and it stops only when we have an ideal triangulation of $\open{M}$ where the only closed normal surface isotopic to a vertex-linking surface is itself a vertex-linking surface.\end{proof}

We are now ready to prove Theorem \ref{thm:bdry-eff}. 
\begin{proof}

 Suppose $\T$ is a triangulation of $M$.  By Theorem 7.1 \cite{jac-rub-0eff}, there is an algorithm to modify the triangulation $\T$  to an ideal triangulation $\T'$ of $\open{M}$ and by Theorem \ref{thm:ideal-bdry-eff}, we can modify the ideal triangulation $\T'$ of $\open{M}$ to an ideal triangulation $\T^*$ of $\open{M}$ so that a (closed) normal surface isotopic to a vertex-linking surface is normally isotopic to that vertex-linking surface. The triangulation $\T^*$ is end-efficient.  
 
Construct any inflation, say $\widehat{\T}$, of the ideal triangulation $\T^*$. Then by Corollary \ref{vertex-link is bdry-link}, the triangulation  $\widehat{\T}$ of $M$ has the property that a closed normal surface in  $\widehat{\T}$  that is isotopic into $\bdy M$ is a boundary-linking surface. So, $\widehat{\T}$  is $\bdy$-efficient. From our observations above, the triangulation  $\widehat{\T}$can not have a normal disk that is not vertex-linking ( $\widehat{T}$ is 0--efficient) and can not have a normal annulus with essential boundary that is not thin edge-linking ( $\widehat{T}$ is annular-efficient).\end{proof}

\section{boundary slopes of surfaces}	If $S$ is a surface and $\gamma$ is a closed curve in $S$, then we call the isotopy class of $\gamma$ a \emph{slope} and refer to it as the slope of $\gamma$.  It follows from the proof of Proposition 3.2 of \cite{jac-rub-sed}  that if $M$ is a link-manifold (nonempty boundary and each boundary component is a torus) and $M$ has no essential annuli between distinct boundary components, then for any $0$--efficient triangulation $\T$ of $M$, there are only finitely many boundary slopes for normal surfaces of a bounded Euler characteristic. Hence, for such an $M$ there are only finitely many boundary slopes for incompressible and $\bdy$--incompressible surfaces of bounded Euler characteristic. We generalize this to general (arbitrary genera for boundary components) compact, irreducible, $\bdy$--irreducible, and an-annular 3-manifolds.

First, we have the following lemma.

\begin{lem}\label{normal-sum-edge-linking-annulus}  Suppose $M$ is a compact 3-manifold with nonempty boundary and $\T$ is a triangulation of $M$.  Furthermore, suppose $F'$ is a normal surface and $A$ is a thin edge-linking annulus about an edge in $\bdy M$. If the Haken sum $F'+A$ is defined, then $F' + A$ is either
\begin{enumerate}\item [(i)] The disjoint union $F'\cup A$, \item[(ii)] A normal surface $F$ and a vertex-linking surface, or
\item[(iii)] A normal surface $F$ isotopic to $F'$.\end{enumerate}\end{lem}
\begin{proof} Following a small isotopy of $A$, we have that $F'\cap A$ is at most a finite number of normal spanning arcs running through the normal quads of $A$ (hence, only meeting normal triangles of $F'$).  
If $F' \cap A=\emptyset$, then we have conclusion $(i)$. So, assume $F'\cap A\ne \emptyset$.

\begin{figure}[htbp] \psfrag{A}{\scriptsize
{$A$}}
 \psfrag{F}{\scriptsize{$F'$}}\psfrag{e}{\tiny{$e\in \bdy M$}}
 \psfrag{V}{\scriptsize{$v$}}\psfrag{D}{\scriptsize{$D$}}\psfrag{G}{\scriptsize{$F$}}
 \psfrag{1}{\scriptsize{$F'+A = D+F$}}\psfrag{2}{\scriptsize{$F'+A = F \sim F'$}}\psfrag{E}{\scriptsize{$F'+A=F$}}

        \begin{center}
  \includegraphics[width=3.5 in]{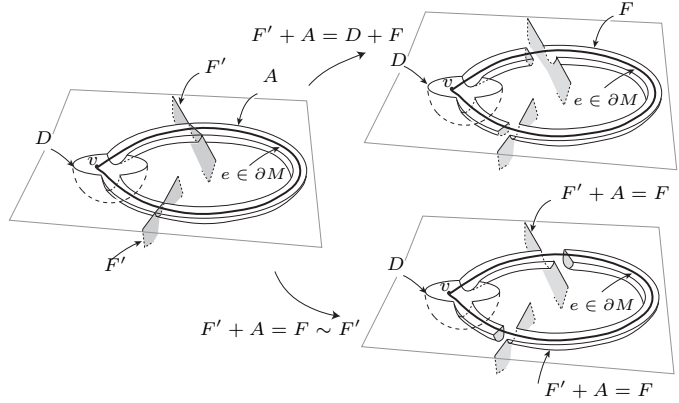}
\caption{Haken sum of normal surface with thin edge-linking annulus. } \label{fr-thin-annulus-Haken-sum.eps}
\end{center}
\end{figure}
 
 An arc $\alpha\subset F'\cap A$ cuts off a small disk $d_{\alpha}$ in $F$ where $\bdy d = \alpha\cup\beta$ with $\beta\subset\bdy M$.  Following a regular exchange along $\alpha$, a copy of $d_\alpha$ is joined with $A$ forming a $\bdy$-compression of $A$. See Figure \ref{fr-thin-annulus-Haken-sum.eps}. Since $F'+A$ is a normal surface, it is only possible that two adjacent such $\bdy$--compressions on $A$ occur so that after normal addition on $A$ they cut off the vertex-linking disk. See the top right-hand drawing in Figure \ref{fr-thin-annulus-Haken-sum.eps}. This gives possibility (ii) of our conclusion.  Otherwise, we get possibility (iii). See bottom right-hand drawing in Figure \ref{fr-thin-annulus-Haken-sum.eps}. \end{proof}

\begin{thm} Suppose $\T$ is an annular-efficient triangulation of the compact 3--manifold $M$.  Then there are only finitely many boundary slopes for connected normal surfaces in $\T$ of bounded Euler characteristic.\end{thm} 

\begin{proof}  A triangulation $\T$ determines a collection of normal surfaces. Among these is a unique collection of normal surfaces (the fundamental normal surfaces) $F_1,\ldots, F_P,$ $T_1,\ldots,T_Q,$ $ A_1,\ldots, A_R$ such that any normal surface $F$ can be written as a Haken sum $F = \sum_1^P p_k F_k +\sum_1^Q q_m T_m +\sum_1^R r_n A_n$, where $p_k, 1\le k\le P; q_m, 1\le m\le Q;$ and $r_n, 1\le n\le R$ are nonnegative integers, $\chi(F_k)<0$, $T_m$ a torus or Klein bottle, and $A_n$ an annulus. Since $\T$ is annular-efficient, $M$ is an-annular and, hence, no $A_m$ can be a M\"obius band.  If we set $F'' = \sum_1^P p_k F_k$, then we have $\chi(F) = \chi(F'')$ and observe for surfaces $F$ with bounded Euler characteristic, there can be only finitely many sums $F'' = \sum_1^R p_k F_k$ of bounded Euler characteristic. 

Hence, for surfaces $F'$ of bounded Euler characteristic, there are at most a finite number of boundary slopes for surfaces of the form   $F' = \sum_1^P p_k F_k +\sum_1^Q q_m T_m$. However, for any connected surface $F = F'+A_r$, only one of the possibilities in Lemma \ref{normal-sum-edge-linking-annulus} can hold and that is (iii). It follows that for $F$ connected and $F = F' + \sum_1^R r_n A_n$, we have $F\sim F'$, $F$ isotopic to $F'$, and so $F$ and $F'$ have the same boundary slopes. This proves our theorem.\end{proof}

The following corollary is immediate as an incompressible and $\bdy$--incompressible surface in an irreducible and $\bdy$--irreducible 3--manifold is isotopic to a normal surface in any triangulation.

\begin{cor} Suppose $M\ne \mathbb{B}^3$ is a compact,
irreducible, $\bdy$--irreducible,  an-annular $3$--manifold. Then there are only finitely many boundary slopes for connected, incompressible, and $\bdy$--incompressible  surfaces in $M$ of bounded Euler characteristic
\end{cor}

\section{Summary}

 In \cite{lack-HS-genus} and  following the work of Thurston (see \cite{Morgan-Smith-conj}),  Epstein and Penner \cite{epstein-penner-ideal}, and Kojima \cite{kojima}, it is shown for $M$  a compact, orientable, simple  (irreducible and $\bdy$-irreducible, an-annular, and atoridal) 3-manifold with boundary, no component a 2-sphere, then $\open{M}$ admits a partially flat angled ideal triangulation.  A partially flat angled ideal triangulation has no normal surfaces with positive Euler characteristic (no normal 2-spheres) and so is a $0$-efficient triangulation; similarly, the only normal surfaces with $0$ Euler characteristic are vertex-linking normal tori and so is a 1-efficient ideal triangulation.  We have not been able to prove that our combinatorial methods give a 1-efficient ideal triangulation. For compact 3-manifolds with only tori boundary, a partially flat angled ideal triangulation is also end-efficient. Using our methods and inflating the partially flat angled ideal triangulation, we arrive at a 1-efficient triangulation of the link manifold.  
 
 For higher genera boundary, one can use our methods for modifying a partially flat angled ideal triangulation to an end-efficient triangulation and by inflating, get an annular-efficient triangulation of the compact simple 3-manifold with nonempty boundary, no component of which is a 2-sphere.

\bibliographystyle{plain}
\bibliography{arXiv-efficient-triangulations-boundary-slopes}

\end{document}